\documentclass[11pt]{amsart}

\usepackage{cetcs-ex-macros}
\usepackage{textcomp}
\newcommand{\xycenterm}[2][=2em]{\vcenter{\hbox{\xymatrix@#1{#2}}}}
\newcommand{\xycentert}[2][=2em]{\[\xycenterm[#1]{#2}\]}
\SelectTips{cm}{}

\title{Exact completion and constructive theories of sets}
\author{Jacopo Emmenegger}
\address{Department of Mathematics, Stockholm University, Sweden.}
\curraddr{DIMA, Università degli Studi di Genova, via Dodecaneso 35, 16146 Genova, Italy}
\email{jacopo.emmenegger@edu.unige.it}
\author{Erik Palmgren$^\dagger$}
\address{Department of mathematics, Stockholm University, Sweden.}
\thanks{{\itshape Acknowledgements.}
The research presented in this paper was supported by the VR grant 2015-03835 from the Swedish Research Council,
and presented at the Logic Colloquium in Leeds in August 2016 and at the XXVI AILA meeting in Padua in September 2017.
We thank the organisers of both events for giving us the opportunity to present our work.
The first author gratefully acknowledge support from the Association of Symbolic Logic
and the G.S.\ Magnusons Foundation to participate in the first event.
We also thank the anonymous referee for a careful reading
and useful comments that improved the shape of the paper.}
\thanks{$^\dagger$ 1963--2019}

\begin{document}

%--------------------------------------------------------------------------------------------------------------------------------------------
\begin{abstract}
In the present paper we use the theory of exact completions to study categorical properties of small setoids in \mltt and,
more generally, of models of the Constructive Elementary Theory of the Category of Sets,
in terms of properties of their subcategories of choice objects (\ie objects satisfying the axiom of choice).
Because of these intended applications, we deal with categories that lack equalisers and just have weak ones,
but whose objects can be regarded as collections of global elements.
In this context, we study the internal logic of the categories involved,
and employ this analysis to give a sufficient condition for the local cartesian closure of an exact completion.
Finally, we apply this result to show when an exact completion produces a model of CETCS.
\end{abstract}
%--------------------------------------------------------------------------------------------------------------------------------------------

\maketitle

%--------------------------------------------------------------------------------------------------------------------------------------------
\section{Introduction}
\label{sec:intro}

Following a tradition initiated by Bishop~\cite{Bi1967},
the constructive notion of set is taken to be a collection of elements together
with an equivalence relation on it, seen as the equality of the set.
In \mltt this is realised with the notion of setoid,
which consists of a type together with a type-theoretic equivalence relation on it~\cite{PaWi2014}.
An ancestor of this construction can be found
in Gandy's interpretation of the extensional theory of simple types into the intensional one~\cite{Ga1956}.
A category-theoretic counterpart is provided by the exact completion construction \exct,
which freely adds quotients of equivalence relations to a category \catct with (weak) finite limits~\cite{CaMa1982,CaVi1998}.
As shown by Robinson and Rosolini, and further clarified by Carboni,
the effective topos can be obtained using this construction~\cite{RoRo1990,Ca1995}.
The authors of~\cite{BCRS1998} then advocated the use of exact completions as an abstract framework
where to study properties of categories of partial equivalence relations,
which are widely used in the semantics of programming languages.
For these reasons, this construction has been extensively studied and has a robust theory~\cite{GrVi1998,CaRo2000,vdB2005,vdBMo2008},
at least when \catct has finite limits,
whereas its behaviour is less understood when \catct is only assumed to have weak finite limits.

The relevance of the latter case comes from the fact that
setoids in \mltt arise as the exact completion of the category of closed types,
which does have finite products but only weak equalisers (what we shall call a \qcart category),
meaning that a universal arrow exists but not necessarily uniquely.
However, this category of types has some other features:
it validates the axiom of choice
and it has a proof-relevant internal logic with a strong existential quantifier.
These features have been investigated by the second author in~\cite{Pa2004},
where this internal logic is called categorical BHK-interpretation.

More generally, the same situation arises for any model of the Constructive Elementary Theory of the Category of Sets (CETCS),
a first order theory introduced by the second author in~\cite{Pa2012a}
in order to formalise properties of the category of sets in the informal set theory used by Bishop.
In fact, this theory provides a finite axiomatisation of the theory of
well-pointed locally cartesian closed pretoposes with enough projectives and a \nno.
Therefore, any model \catet of CETCS is the exact completion of its projective objects,
which form a \qcart category \catpt.
As for closed types in \mltt, these are objects satisfying a categorical version of the axiom of choice, and
the internal logic of \catet on the projectives is (isomorphic to) the categorical BHK-interpretation
of intuitionistic first order logic in \catpt.

The aim of the present paper is to isolate certain properties of a \qcart category \catct
that will ensure that its exact completion is a model of CETCS
while, at the same time, making sure that these properties are satisfied by the category of closed types in \mltt.
In fact, for some of the properties defining a model \catet of CETCS, an equivalent formulation in terms of projectives of \catet
is already known, as in the case of pretoposes~\cite{GrVi1998},
or follows easily from known results, as for \nnos~\cite{Ca1995,BCRS1998}.
However, in the general case of weak finite limits (or just \qcart categories),
a complete characterisation of local cartesian closure in terms of a property of the projectives is still missing.

The first contribution of this paper consists of a condition on a category
which is sufficient for the local cartesian closure of its exact completion.
This condition is a categorical formulation of Aczel's Fullness Axiom
from Constructive Zermelo-Fraenkel set theory (CZF)~\cite{Ac1978,AcRa2001},
and it is satisfied by the category of closed types.
A complete characterisation of local cartesian closure
for an exact completion is given by Carboni and Rosolini in~\cite{CaRo2000},
but it has been recently discovered that the argument used
requires the projectives to be closed under finite limits~\cite{Em2018a}.
Another sufficient condition, which applies to those exact completions arising from certain homotopy categories,
has been recently given by van den Berg and Moerdijk~\cite{vdBMo2018}.
We formulate our notion of Fullness, and the proof of local cartesian closure, in the context of well-pointed \qcart categories
in order to match some aspects of set theory, like extensionality.
However, a suitably generalised version of our formulation of Fullness in fact
reduces to Carboni and Rosolini's characterisation in the presence of finite limits,
and is tightly related to van den Berg and Moerdijk's condition as well~\cite{Em2018b}.

In CZF minus Subset Collection, the Fullness Axiom is equivalent to Subset Collection.
Hence it is instrumental in the construction of Dedekind real numbers in CZF and it implies Exponentiation~\cite{AcRa2001}.
The axiom states the existence of a full set $ F $ of total relations (\ie multi-valued functions) from a set $ A $ to a set $ B $,
where a set $ F $ is full if every total relation from $ A $ to $ B $ has a subrelation in $ F $,
\ie if $ F \subseteq \text{TR}(A,B) $ and
\[
  \forall R \in \text{TR}(A,B) \	\exists S \in F \	S \subseteq R,
\]
where $ \text{TR}(A,B) \coloneqq \Set{ R \subseteq A \times B | \forall a \in A \, \exists b \in B \ (a,b) \in R } $
is the class of total relations from $ A $ to $ B $.
Since functional relations are minimal among total relations, a full set must contain all graphs of functions,
however it is not a (weak) exponential as it may also contain non-functional relations.
We shall use a characterisation of local cartesian closure
in terms of closure under families of partial functional relations as in~\cite{Pa2012a} and,
similarly, we shall formulate a version of the Fullness Axiom in terms of families of partial pseudo-relations (\ie non-monic relations).
The key aspect of the proof is the very general universal property of a full set (or of a full family of partial pseudo-relations),
which endows the internal (proof-relevant) logic with implication and universal quantification.

The second contribution of the paper is a complete characterisation of well-pointed exact completions in terms of their projectives.
We relate well-pointedness, which amounts to extensionality with respect to global elements, with certain choice principles,
namely versions of the axiom of unique choice in \exct and the axiom of choice in \catct.
We also exploit this correspondence to simplify the internal logic of the categories under consideration,
and the exact completion construction itself.
In the related context of quotient completions of elementary doctrines,
an analogous result relating choice principles is obtained by Maietti and Rosolini in~\cite{MaRo2016}.

The paper is understood as being formulated in an essentially algebraic theory for category theory over intuitionistic first order logic,
as the one presented in~\cite{Pa2012a}.
However,
we believe that all the results herein can be formalised in intensional \mltt using E-categories~\cite{PaWi2014},
and this is indeed the case for those regarding the category of setoids.
A step towards this goal is made in~\cite{Pa2016},
where CETCS is formulated in a dependently typed first-order logic,
which can be straightforwardly interpreted in \mltt.

The paper is organised as follows.
In \cref{sec:excat} we recall the basic category-theoretic concepts needed in the paper
and provide a brief overview of already known facts about the exact completion construction.
% the exact completion construction.
% Here we recall the category-theoretic concepts needed to illustrate the exact completion construction,
% define the categories of small setoids and small types in \mltt,
% which will be the main intended examples throughout the paper,
% and relate the setoid construction to the exact completion of small types.

In \cref{sec:elemcat} we consider the concept of elemental category,
which is needed to formulate the constructive version of well-pointedness satisfied by models of CETCS,
and which allows to regard objects as collections of (global) elements.
Indeed, in abstract categorical terminology, it amounts to say that the global section functor is conservative,
however we avoid this formulation since it refers to the category of sets,
and prefer an elementary definition instead.
The main result of this section is a characterisation of elemental exact completions
as those arising from categories satisfying a version of the axiom of choice.

\cref{sec:closure} contains the main result of the paper,
namely our categorical formulation of Aczel's Fullness Axiom
and the proof that it implies the local cartesian closure of the exact completion.
In this section we fully exploit the simplifications of the internal logic
and the exact completion construction given by elementality,
as well as the proof relevance of the internal logic given by the BHK-interpretation.

In \cref{sec:cetcs} we recall the  axioms of CETCS from~\cite{Pa2012a}
and discuss how its models are exact completions of their choice objects (\ie projective objects).
We then use the results from the previous sections, and already known ones,
to show when an exact completion produces a model of CETCS.

Finally, \cref{sec:setoids} treats the case of setoids in \mltt.
Here we recall the main concepts and definitions %and the properties of the category of setoids,
and show how to apply results from the previous section to obtain a compact and uniform proof
that the category of setoids is a model of CETCS.

%--------------------------------------------------------------------------------------------------------------------------------------------
\section{Exact and \qcart categories}
\label{sec:excat}

An \emph{equivalence relation} in a category \catct with finite limits
is a subobject $ r \colon R \mono X \times X $ such that there are
(necessarily unique) arrows witnessing reflexivity, symmetry and transitivity as in the following diagrams
\begin{equation}
\label{eq:eqrel}
\xycenterm[C=3em]{
						&	R	\ar[d]^-r	&%
							&	R	\ar[d]^-r	&%
									&	R	\ar[d]^-r	\\
	X	\ar[r]_-{\diag_X} \ar@{.>}[ur]^-{\rho}	&	X \times X	&%
	R	\ar[r]_-{\pbk{r_2,r_1}} \ar@{.>}[ur]^-{\sigma}	&	X \times X	&%
	R \times_X R	\ar[r]_-{\pbk{r_1 p_1, r_2 p_2}} \ar@{.>}[ur]^-{\tau}	&	X \times X	}
\end{equation}
where $ R \xleftarrow{\,p_1\,} R \times_X R \xrightarrow{\,p_2\,} R $
is a pullback of $ R \xrightarrow{\,r_2\,} X \xleftarrow{\,r_1\,} R $.
Here and in the rest of the paper we denote by $ f_i $ the $ i $-th component of an arrow $ f $ into a product.
Subobjects obtained by pulling back an arrow along itself are always equivalence relations,
these are called \emph{kernel pairs}.
A diagram of the form $ R \psrel X \to Y $ is \emph{exact}
if it is a coequaliser diagram and $ R \psrel X $ is the kernel pair of $ X \to Y $.
In such a situation,
the regular epi $ X \to Y $ is called \emph{quotient} of the equivalence relation $ R \mono X \times X $.

\begin{defin}
A category is \emph{exact} if it has finite limits,
and pullback-stable quotients of equivalence relations.
An exact category is a \emph{pretopos} if it has disjoint and pullback-stable finite sums,
and the initial object is strict.
\end{defin}

In an exact category,
regular epis are quotients of their kernel pair
and they coincide with the simpler notion of covers:
an arrow $ f $ is a \emph{cover} if,
whenever it factors as $ f = gh $ with $ g $ monic, then $ g $ is in fact an iso.
In a pretopos regular epis also coincide with epis.

We say that an object P is \emph{projective} if,
for every cover $ X \to Y $ and every arrow $ P \to Y $,
there is $ P \to X $ such that the obvious triangle commutes.

\begin{defin}
A \emph{projective cover} of an object $ X \in \catcm $ is given by a projective object $ P $ and a cover $ P \to X $.
A \emph{projective cover} of \catct is a full subcategory \catpt of projective objects such that
every $ X \in \catcm $ has a projective cover $ P \in \catpm $.
\catct has \emph{enough projectives} if it has a projective cover.
\end{defin}

The property of having enough projectives corresponds to the set-theoretic principle known as
Presentation Axiom~\cite{Ac1978,AcRa2001},
which states that every set is the image of a choice set (or base),
\ie a set for which the axiom of choice holds.

Projective covers are not necessarily closed under limits that may exist in \catct.
However they do have a weak limit of every diagram that has a limit in \catct~\cite{CaVi1998},
where a weak limit is defined in the same way as a limit
but dropping uniqueness of the universal arrow.
Indeed, if $ L \in \catcm $ is a limit in \catct of a diagram $ \mathcal{D} $ in \catpt,
then any projective cover $ P \in \catpm $ of $ L $ is a weak limit of $ \mathcal{D} $ in \catpt:
given any cone over $ \mathcal{D} $ with vertex $ Q \in \catpm $, the weak universal $ Q \to P $ is obtained
lifting the universal $ Q \to L $ along the cover $ P \to L $ using projectivity of $ Q $.

Nevertheless, in the rest of the paper we shall be interested in subcategories of projectives which are closed under finite products,
so we introduce the following definition.

\begin{defin}
A category is \emph{\qcart} if it has finite products and weak equalisers.
\end{defin}

\begin{rem}
As for the case of limits,
a category with finite products has all weak finite limits \iff it has weak equalisers \iff it has weak pullbacks.
\end{rem}

\begin{rem}
\label{rem:bhklog}
\Qcart categories come naturally equipped with a proof-relevant internal logic.
This interpretation has been investigated in detail by the second author in~\cite{Pa2004},
where it is called \emph{categorical BHK-interpretation}
due to its similarities to the propositions-as-types correspondence.
Lawvere gave a description of it in the case of locally cartesian closed categories
as an instance of what he calls the Curry-L\"auchli adjoint,
\ie the adjunction between the bicategory of posets and the bicategory of categories~\cite{La1996}.
This logic may also be understood in terms of hyperdoctrines~\cite{La1969},
and this perspective has been developed by Grandis~\cite{Gr1997,Gr2000}
and more recently by Maietti and Rosolini~\cite{MaRo2013a,MaRo2013b}
with the name ``weak subobjects''.
Since this internal logic will be one of the main tool in the proof of our main result,
we briefly review it here.

Recall that, given two arrows $ f \colon Y \to X $ and $ g \colon Z \to X $,
$ f \leq g $ means that there is $ h \colon Y \to Z $ such that $ g h = f $.
This defines a preorder on $ \catcm/X $,
we denote by $ \presubc(X) $ its order reflection
and, following~\cite{Pa2004}, call its elements \emph{presubobjects}.
Presubobjects are used for the interpretation of predicates.
Since weak limits are unique up to presubobject equivalence,
weak pullbacks can be used to interpret weakening and substitution.
For the same reason, we can interpret equality with weak equalisers and conjunction with weak pullbacks,
while postcomposition provides an interpretation for the existential quantifier.
Hence regular logic has a sound interpretation into any \qcart category.
\end{rem}

As shown by Carboni and Vitale~\cite{CaVi1998},
any category with weak finite limits \catct can be regarded as a projective cover of an exact category,
known as the exact completion of \catct.
This construction consists in freely adding quotients of pseudo-equivalence relations
and we now describe it in the case of a \qcart category \catct.

Objects of the \emph{exact completion} \exct are pseudo-equivalence relations in \catct,
that is arrows $ r \colon R \to X \times X $
such that there are (not necessarily unique) arrows for reflexivity, symmetry and transitivity as in \eqref{eq:eqrel},
where now the domain of $ \tau $ is just a weak pullback of $ r_1 $ and $ r_2 $.
Arrows from $ R \to X \times X $ to $ S \to Y \times Y $ in \exct are
arrows $ f \colon X \to Y $ in \catct
such that there is $ \hat{f} \colon R \to S $ making the left-hand diagram below commute,
and where $ f, g \colon X \to Y $ are identified in \exct
if there is $ h \colon X \to S $ making the right-hand diagram below commute.
% \textbf{EITHER:}
% pairs $ (f,\hat{f}) $ of arrows $ f \colon X \to Y $ and $ \hat{f} \colon R \to S $ in \catct
% such that the left-hand diagram below commutes,
% and two such pairs are equal in \exct
% if there is $ h \colon X \to S $ making the right-hand diagram below commute.
% \textbf{OR:}
% equivalence classes $ [f] $ of arrows $ f \colon X \to Y $ in \catct
% such that there is an arrow $ \hat{f} \colon R \to S $ making the left-hand diagram below commute,
% and where $ f, g \colon X \to Y $ are equivalent
% if there is $ h \colon X \to S $ making the right-hand diagram below commute.
%
\xycentert[C=3em]{
	R	\ar[d] \ar[r]^-{\hat{f}}		&	S	\ar[d]	&&%
													&	S	\ar[d]	\\
	X \times X	\ar[r]^-{f \times f}	&	Y \times Y	&&%
							X	\ar@/^/[ur]^-h \ar[r]^-{\pbk{f,g}}	&	Y \times Y	}

The functor $ \Gamma \colon \catcm \to \excm $ mapping an object $ X \in \catcm $ to the diagonal on $ X $
is full and faithful and preserves all the finite limits which exist in \catct.
The image of this embedding is a projective cover of \exct,
and every exact category with enough projectives is in fact an exact completion.
In the case of \qcart categories, this characterisation assumes the following form.

\begin{theor}[\cite{CaVi1998}]
\label{theor:carvit}
Every exact category with a \qcart projective cover is the exact completion of a \qcart category,
namely its subcategory of projectives.
Conversely, every \qcart category appears as a projective cover of its exact completion.
\end{theor}

\begin{rem}
The theory of exact completions provides an isomorphism of posets
\[
  \text{Sub}_{\excm}(\Gamma X) \cong \presub{\catcm}(X),
\]
which, in particular, commutes with the regular-logic structure on both posets~\cite{CaVi1998}.
Hence it guarantees that the internal logic of \exct on a projective is still the BHK-interpretation in \catct of intuitionistic logic.
\end{rem}

%--------------------------------------------------------------------------------------------------------------------------------------------
\section{Elemental categories}
\label{sec:elemcat}

An object $ G $ in a category \catct is called a \emph{strong generator} if an arrow $ f \colon X \to Y $ is an iso whenever
\[	(\forall y \colon G \to Y)(\exists! x \colon G \to X)\, f x = y.	\]
An object $ G $ is \emph{separating} if for any pair of arrows $ f,g \colon X \to Y $, $ f = g $ whenever
\[	(\forall x \colon G \to X)\, f x = g x.	\]

\begin{defin}
A category \catct with a terminal object is \emph{elemental} if the terminal object is a strong generator and is separating%
\end{defin}

The terminology comes from the fact that objects in elemental categories can be regarded,
to a certain extent, as collections of global elements.
In particular, this simplifies the internal logic of an elemental category, as shown in \cref{prop:welemlogic,corol:elemlog,prop:elemlogic}.

We denote global elements $ x \colon \termm \to X $ as $ x \in X $ and simply call them elements.
Moreover, if $ f \colon X \to Y $ and $ y \in Y $, we write $ y \inar f $ if there is $ x \in X $ such that $ fx = y $.
An arrow $ f \colon X \to Y $ is \emph{injective} if $ (\forall x,x' \in X)( f x = f x' \limpld x = x' ) $,
while it is \emph{surjective} if $ (\forall y \in Y)\, y \inar f $.
Notice that the terminal object is a strong generator \iff an arrow is iso exactly when it is injective and surjective.
Finally, an object $ Y $ in \catct is a \emph{choice object}
if every surjection $ f \colon X \to Y $ has a section,
\ie an arrow $ g \colon Y \to X $ such that $ f g = \id_Y $.

\begin{exmp}
Sets and functions in (a model of) ZF form an elemental category.
Moreover, the Axiom of Choice is equivalent to the fact that every object is a choice object,
\ie every surjective function has a section.
\end{exmp}

In the following lemma we collect some immediate results.

\begin{lem}
\label{lem:elemconseq}
Let \catct be a category with a terminal object.
\begin{enumerate}[label={\itshape(\roman*)}]
\item	\label{elemconseq:sep}
	If the terminal object is separating, then every surjection is epic and every injection is monic.
\item	\label{elemconseq:proj}
	The terminal object is projective \iff every cover is surjective.
\item	\label{elemconseq:strgen}
	If the terminal object is a strong generator, then every surjection is a cover.
	The converse holds if every injection is monic.
\end{enumerate}
\end{lem}

In the presence of weak equalisers,
we can derive elementality from a categorical choice principle.

\begin{lem}
\label{lem:weqelem}
Let \catct be a \qcart category.
If every object is a choice object, then \catct is elemental.
\end{lem}
\begin{proof}
Since every surjection has a section, it follows that every surjection is a cover.
Therefore it is enough to show that the terminal object is separating,
since elementality will follow from \ref{lem:elemconseq}\ref{elemconseq:strgen}.
Let $ f,g \colon X \to Y $ be such that $ f x = g x $ for every $ x \in X $,
and let $ e \colon E \to X $ be a weak equaliser for $ f $ and $ g $.
Since $ f $ and $ g $ coincide on elements, $ e $ is surjective,
hence it has a section $ s \colon X \to E $.
Therefore $ f = f e s = g e s = g $ as required.
\end{proof}

Since an equaliser is the same as a monic weak equaliser,
with a similar argument we can also prove the following.

\begin{lem}
\label{lem:eqelem}
Let \catct be a category with finite limits.
Then \catct is elemental \iff every surjection is a cover.
\end{lem}

The following result proves that, in every \qcart category,
extensionality of presubobjects is equivalent to a categorical choice principle.

\begin{prop}
\label{prop:welemlogic}
Let \catct be a \qcart category
and consider the following.
\begin{enumerate}[label={\itshape(\roman*)}]
\item	\label{welemlogic:elem}
	Every surjection has a section.
\item	\label{welemlogic:ext}
	For every object $ X $ and arrows $ a,b $ with codomain X,
	\[	a \leq b	\qquad \text{\iff} \qquad	(\forall x \in X)( x \inar a \limpld x \inar b ).	\]
\item	\label{welemlogic:ac}
	For every pseudo-relation $ r \colon R \to X \times Y $
	\[	(\forall x \in X)(\exists y \in Y)\, \pbk{x,y} \inar r	\, \limpld \,%
							(\exists f \colon X \to Y)(\forall x \in X)\, \pbk{x, f x} \inar r.	\]
\end{enumerate}
Statements \ref{welemlogic:elem} and \ref{welemlogic:ext} are equivalent and imply statement \ref{welemlogic:ac}.
If the terminal object is separating, then they are equivalent.
\end{prop}
\begin{proof}
\ref{welemlogic:elem} $ \Rightarrow $ \ref{welemlogic:ext}
The direction from left to right always holds,
so let us assume that $ (\forall x \in X)( x \inar a \limpld x \inar b ) $
and observe that it amounts to the surjectivity of any weak pullback of $ b $ along $ a $.
Hence there is a section of it which, in turn, yields an arrow witnessing $ a \leq b $.

\ref{welemlogic:ext} $ \Rightarrow $ \ref{welemlogic:elem}
This follows from the fact that an arrow $ f \colon X \to Y $ is surjective precisely when
$ (\forall y \in Y)( y \inar id_Y \limpld y \inar f ) $ and that any arrow witnessing $ id_Y \leq f $ is a section of $ f $.

\ref{welemlogic:elem} $ \Rightarrow $ \ref{welemlogic:ac}
Immediate from the fact that $ (\forall x \in X)(\exists y \in Y)\, \pbk{x,y} \inar r $ amounts to surjectivity of $ r_1 \colon R \to X $.

\ref{welemlogic:ac} $ \Rightarrow $ \ref{welemlogic:ext}
As before, we only need to show the direction from right to left.
Let $ r \colon R \to A \times B $ be given as a weak pullback of $ a $ and $ b $ and observe that
$ (\forall x \in X)( x \inar a \limpld x \inar b ) $ implies $ (\forall u \in A)(\exists v \in B)\, \pbk{u,v} \inar r $.
Therefore we obtain an arrow $ f \colon A \to B $ such that $ b f u = a u $ for every $ u \in A $.
If the terminal object is separating, this implies $ a \leq b $ as required.
\end{proof}

In the presence of finite limits we have the following.

\begin{prop}
\label{prop:elemlogic}
Let \catct be a category with finite limits,
and consider the following.
\begin{enumerate}[label={\itshape(\roman*)}]
\item	\label{elemlogic:elem}
	Every surjective mono has a section.
\item	\label{elemlogic:ext}
	For every object $ X $ and monos $ a,b $ with codomain $ X $,
	\[	a \leq b \qquad \text{\iff} \qquad (\forall x \in X)( x \inar a \limpld x \inar b ).	\]
\item	\label{elemlogic:auc}
	For every relation $ r \colon R \mono X \times Y $
	\[	(\forall x \in X)(\exists ! y \in Y)\, \pbk{x,y} \inar r	\,\limpld\,%
							(\exists f \colon X \to Y)(\forall x \in X)\, \pbk{x, f x} \inar r.	\]
\end{enumerate}
Statements \ref{elemlogic:elem} and \ref{elemlogic:ext} are equivalent,
and imply statement \ref{elemlogic:auc}.
If the terminal object is separating, then they are equivalent.
\end{prop}

\begin{rem}
In the presence of finite limits,
\cref{lem:eqelem} tells us that elementality is equivalent to item \ref{elemlogic:elem} in \cref{prop:elemlogic}.
Hence \cref{prop:elemlogic} generalises Propositions 4.3 and 4.4 in~\cite{Pa2012a},
where only the implications from elementality to \ref{elemlogic:ext} and \ref{elemlogic:auc} are proved.
As shown in \cref{prop:elemlogic}, if the terminal object is separating,
these are in fact equivalent to elementality.
% 
% In particular, 
% we can interpret it as saying that, for an elemental category,
% extensionality of subobjects, \ie item \ref{elemlogic:ext},
% is equivalent to the so-called axiom of unique choice, \ie item \ref{elemlogic:auc}.
% In addition, as soon as the terminal object is separating,
% either of them actually characterises elemental categories.
\end{rem}

An immediate consequence of \cref{prop:welemlogic} is that the internal logic of an elemental \qcart category is determined,
up to presubobject equivalence, by global elements.
Here we distinguish internal connectives and quantifiers by adding a dot on top of them
and, to increase readability, we commit the common abuse of dealing with representatives instead of actual presubobjects.
A similar result for the usual categorical interpretation of logic is in Theorem 5.6 in~\cite{Pa2012a},
which can be seen as a consequence of \cref{prop:elemlogic}.

\begin{corol}
\label{corol:elemlog}
Let \catct be a \qcart category where every object is a choice object,
and let $ a,b \in \presubc(X) $, $ f,g \colon Y \to X $ and $ r \in \presubc(X \times Y) $, then:
\begin{enumerate}[label={\itshape(\roman*)}]
\item	$ y \inar \presubpb{f}a \ $ \iff $ \ f y \inar a $,
\item	$ y \inar (f \overset{\text{\bfseries .}}{=} g) \ $ \iff $ \ f y = g y $,
\item	$ x \inar (a \overset{\text{\bfseries .}}{\land} b) \ $ \iff $ \ x \inar a \land x \inar b $,
\item	$ x \inar \overset{\text{\bfseries .}}{\exists}_Y r \ $ \iff $ \ (\exists y \in Y)\, \pbk{x,y} \inar r $,
\end{enumerate}
and the presubobjects obtained by
$ \presubpb{f},\overset{\text{\bfseries .}}{=},\overset{\text{\bfseries .}}{\land} $ and $ \overset{\text{\bfseries .}}{\exists} $
are uniquely determined by the universal closure of the previous relations.
\end{corol}
\begin{proof}
We prove the statement for $ \overset{\text{\bfseries .}}{\land} $, the other proofs are similar.
Of course if $ c \colon C \to X $ is a representative of $ a \overset{\text{\bfseries .}}{\land} b $,
then the equivalence in \textit{(iii)} must hold for every $ x \in X $.
Conversely, suppose that $ (\forall x \in X)( x \inar c \liffld x \inar a \land x \inar b ) $
and let $ p \colon P \to X $ be a representative of $ a \overset{\text{\bfseries .}}{\land} b $ (\eg a weak pullback of $ a $ and $ b $).
Then $ x \inar c \liffld x \inar p $ for every $ x \in X $,
so \ref{prop:welemlogic}.\ref{welemlogic:ext} implies that
$ c $ is also a representative of $ a \overset{\text{\bfseries .}}{\land} b $.
\end{proof}

\cref{prop:welemlogic} also allows for a simpler construction of the exact completion
as described in~\cite{CaMa1982,CaVi1998}
in the case of a \qcart category where every object is a choice object.
\Cref{corol:elemeqrel} treats the case of objects
and \cref{corol:elemexfn} that one of arrows.

For a pseudo-relation $ r \colon R \to X \times X $,
denote with $ \sim_r $ the relation induced by $ r $ on the elements of $ X $,
that is
\[	x \sim_r x'	\ \liffld \	\pbk{x,x'} \inar r.	\]

\begin{corol}
\label{corol:elemeqrel}
Let \catct be a \qcart category where every object is a choice object.
Then for every pseudo-relation $ r \colon R \to X \times X $:
\begin{enumerate}[label={\itshape(\roman*)}]
\item	$ r $ is reflexive \iff %$ \sim_r $ is reflexive,
\[	 (\forall x \in X)\, x \sim_r x.	\]
\item	$ r $ is symmetric \iff %$ \sim_r $ is symmetric,
\[	(\forall x, x' \in X)( x \sim_r x' \limpld x' \sim_r x)	\]
\item	$ r $ is transitive \iff %$ \sim_r $ is transitive.
\[	(\forall x,x',x'' \in X)( x \sim_r x' \limpld x' \sim_r x'' \limpld x \sim_r x'' )	\]
\end{enumerate}
\end{corol}
\begin{proof}
This proof and that one of \cref{corol:elemexfn} are just a matter
of unfolding definitions and applying \cref{prop:welemlogic}\ref{welemlogic:ext}.
We prove point \textit{(iii)} in order to exemplify the idea.

One direction is straightforward, so let us assume
\begin{equation}
\label{elemeqrel:trans}
	(\forall x,x',x'' \in X)( x \sim_r x' \land x' \sim_r x'' \limpld x \sim_r x'' ).
\end{equation}
Consider the following weak pullback
\xycentert{
	P	\ar[d]^-{p_1} \ar[r]^-{p_2}	&	R	\ar[d]^-{r_1}	\\
	R	\ar[r]^-{r_2}			&	X			}
and define $ p \coloneqq \pbk{r_1 p_1,r_2 p_2} \colon P \to X \times X $.
According to the construction described in~\cite{CaVi1998},
proving transitivity of $ r $ amounts to show that $ p \leq r $.
Thanks to \ref{prop:welemlogic}.\ref{welemlogic:ext} it is enough to show that
$ \pbk{x,x''} \inar p \limpld \pbk{x,x''} \inar r$.
But $ \pbk{x,x''} \inar p $ implies that there is $ x' \in X $ such that $ x \sim_r x' $ and $ x' \sim_r x'' $,
hence $ \pbk{x,x''} \inar r $ from \eqref{elemeqrel:trans}.
\end{proof}

\begin{corol}
\label{corol:elemexfn}
Let \catct be a \qcart category where every object is a choice object
and let $ r \colon R \to X \times X $ and $ s \colon S \to Y \times Y $ be two pseudo-equivalence relations.
\begin{enumerate}[label={\itshape(\roman*)}]
\item	\label{elemexfn:ext}
	An arrow $ f \colon X \to Y $ gives rise to an arrow $ r \to s $ in \exct \iff,
				\[	(\forall x,x' \in X)( x \sim_r x' \, \limpld \, f x \sim_s f x' ).	\]
\item	\label{elemexfn:eq}
	Two arrows $ f,g \colon r \to s $ in \exct are equal in \exct \iff,
				\[	(\forall x \in X)\, f x \sim_s g x .	\]
\end{enumerate}
\end{corol}

Finally, we can prove that elemental exact completions are precisely those exact categories
with a projective cover consisting of choice objects.
In light of the equivalences in \cref{prop:welemlogic,prop:elemlogic},
this result should be compared with the equivalence, proved in~\cite{MaRo2016},
between elementary doctrines satisfying the Axiom of Choice
and elementary quotient completions satisfying the Axiom of Unique Choice.

\begin{theor}
\label{theor:elemex}
Let \catct be a \qcart category.
Then \exct is elemental \iff every object in \catct is a choice object.
\end{theor}
\begin{proof}
Let us first prove that, if \exct is elemental, then every surjection in \catct splits.
Observe that, if $ f \colon X \to Y $ is surjective in \catct, then $ f \colon \Delta_X \to \Delta_Y $ is surjective in \exct,
hence a cover because of elementality.
But $ \Delta_Y $ is projective in \exct, therefore we get a section $ g \colon \Delta_Y \to \Delta_X $
which is a section of $ f $ in \catct as well.

In order to prove the other implication, thanks to \cref{lem:eqelem}
it is enough to show that in \exct every monic surjection has a section.
To this aim, let $ r \colon R \to X \times X $ and $ s \colon S \to Y \times Y $ be two pseudo-equivalence relations in \catct
and let $ f \colon X \to Y $ be such that $ x \sim_r x' \limpld f x \sim_s f x' $ for every $ x,x' \in X $.
Assume that $ f $ is monic and surjective in \exct.
In particular
\[	(\forall y \in Y)(\exists x \in X)\, f x \sim_s y,	\]
hence \ref{prop:welemlogic}.\ref{welemlogic:ac} yields $ g \colon Y \to X $
such that $ f g y \sim_s y $ for all $ y \in Y $.
If $ y \sim_s y' $, then $ f g y \sim_s f g y' $ and injectivity of $ f $ implies $ g y \sim_r g y' $,
so $ g $ is an arrow $ s \to r $ from \ref{corol:elemexfn}.\ref{elemexfn:ext},
and a section of $ f $ in \exct from \ref{corol:elemexfn}.\ref{elemexfn:eq}.
\end{proof}

%-------------------------------------------------------------------------------------------------------------------------------------
\section{Fullness and exponentiation}
\label{sec:closure}
This section contains the main contribution of the paper,
which provides a sufficient condition on the choice objects of an elemental exact completion
that ensures the local cartesian closure of the latter.

We begin recalling a characterisation of local cartesian closure for elemental categories
in terms of closure under families of partial functional relations~\cite{Pa2012a}.

\begin{defin}
\label{def:properties}
Let $ Y \overset{g}{\longrightarrow} X \overset{f}{\longrightarrow} I $ be two arrows
in a category \catct with finite products.
A pair $ h \colon J \to I $, $ r \colon R \to J \times X \times Y $
\begin{enumerate}[label=(\alph*)]
\item	\label{comm}
	is a \emph{family of partial sections of g} if for every $ j \in J $, $ x \in X $ and $ y \in Y $,
		\[	\pbk{j, x, y} \inar r	\ \limpld \	gy = x,	\]
\item	\label{exist}
	has \emph{domains indexed by f} if for every $ j \in J $ and $ x \in X $
		\[	f x = h j	\ \liffld \	(\exists y \in Y)\, \pbk{j, x, y} \inar r,	\]
\item	\label{uniq}
	is \emph{functional} if for every $ j \in J $, $ x,x' \in X $ and $ y,y' \in Y $
		\[	\pbk{j, x, y}\inar r \, \land \, \pbk{j, x', y'} \inar r \, \land \, x = x'	\ \limpld \	y = y'.	\]
\end{enumerate}
\end{defin}

\begin{defin}
\label{def:undepprod}
Let $ Y \overset{g}{\longrightarrow} X \overset{f}{\longrightarrow} I $ be two arrows
in a category \catct with finite limits.
A pair $ h \colon J \to I $ and $ r \colon R \mono J \times X \times Y $ is a
\emph{family of functional relations over $ f,g $} if
it satisfies properties \ref{comm}--\ref{uniq} from \cref{def:properties}.
If $ J $ is terminal, then $ h \in I $ will be called the \emph{domain index} of $ r $
and $ \pbk{r_2,r_3} \colon R \mono X \times Y$ will be called a \emph{functional relation}.

A \emph{universal dependent product for $ f,g $} is a family of functional relations
$ \phi \colon F \to I $ and $ \alpha \colon P \mono F \times X \times Y $ over $ f,g $ such that,
for every functional relation $ r \colon R \mono X \times Y $ over $ f,g $ with domain index $ i \in I $,
there is a unique $ c \in F $ such that $ \phi c = i $ and for all $ x \in X $ and $ y \in Y $
\begin{equation}
\label{undepprod:univ}
	\pbk{c, x, y} \inar \alpha	\ \liffld \	\pbk{x, y} \inar r.
\end{equation}
\end{defin}

\begin{rem}
Intuitively, the property defining a universal dependent product amounts to say that
the arrow $ \phi $ contains a code $ c $ for every functional relation over $ f,g $.
In an elemental category with finite limits, functional relations coincide with arrows
(because of \cref{prop:elemlogic}\ref{elemlogic:auc})
and Theorem 6.8 in~\cite{Pa2012a} proves that, in such a category,
having all universal dependent products is equivalent to local cartesian closure.
\end{rem}

Considering pseudo-relations instead of relations,
and dropping functionality we obtain the following
version of Aczel's notion of full set.

\begin{defin}
\label{def:rfull}
Let $ Y \overset{g}{\longrightarrow} X \overset{f}{\longrightarrow} I $ be arrows in a \qcart category.
A pair $ h \colon J \to I $, $ r \colon R \to J \times X \times Y $ is a \emph{family of pseudo-relations over $ f,g $}
if it satisfies properties \ref{comm} and \ref{exist} from \cref{def:properties}.
If $ J $ is terminal, then $ h \in J $ will be called the \emph{domain index} of $ r $,
and $ \pbk{r_2,r_3} \colon R \to X \times Y $ will just be called a pseudo-relation over $ f,g $.

A \emph{full family of pseudo-relation over $ f,g $} is a family of pseudo-relations
$ \phi \colon F \to I $ and $ \alpha \colon P \to F \times X \times Y $ over $ f,g $ such that,
for every pseudo-relation $ r \colon R \to X \times Y $ over $ f,g $ with domain index $ i \in I $,
there is $ c \in F $ such that $ \phi c = i $ and for all $ x \in X $ and $ y \in Y $
\begin{equation}
\label{rfull:univ}
	\pbk{c,x,y} \inar \alpha	\ \limpld \	\pbk{x,y} \inar r.
\end{equation}

A \qcart category is \emph{closed for pseudo-relations}
if it has a full family of pseudo-relations for any pair of composable arrows.
\end{defin}

Notice that in~\eqref{rfull:univ} only one direction of the implication is required,
as opposed to the bi-implication in~\eqref{undepprod:univ}.
This is to mimic the behaviour of a full set in CZF as defined by Aczel (see~\cite{Ac1978} pg.\ 58 or~\cite{AcRa2001}):
a (total) relation is not necessarily an element of a full set $ F $,
but it contains as subrelation an element of $ F $.

\begin{exmp}
\cref{prop:typrc} in the last section proves that the E-category of types is closed for pseudo-relations.
\end{exmp}

The next two results show that
closure for pseudo-relations endows the internal logic of a \qcart category with
implication and universal quantification.

\begin{lem}
Let \catct be a \qcart category where every object is a choice object.
If \catct is closed for pseudo-relations,
then for every $ f \colon X \to I $ there is a right adjoint to $ \presubpbf \colon \presubc(I) \to \presubc(X) $.
\end{lem}
\begin{proof}
As for weak pullbacks, it is easy to see that full families are unique up to presubobject equivalence.
This defines an order-preserving function $ \forall_f \colon \presubc(X) \to \presubc(I) $.
Given an arrow $ g \colon Y \to X $,
let $ \phi \colon F \to I $ and $ \alpha \colon P \to F \times X \times Y $ be a full family of pseudo-relations over $ f,g $.
We need to show that $ h \leq \phi \liffld \presubpb{f} h \leq g $ for every $ h \colon Z \to I $,
and we shall make use of the statement in \ref{prop:welemlogic}.\ref{welemlogic:ext} in doing so.

Assume $ h \leq \phi $.
We have that $ x \inar \presubpb{f} h $ implies $ f x \inar h \leq \phi $,
so there is $c \in F $ such that $ \phi c = f x $ and,
from \ref{def:undepprod}.\ref{exist},
we obtain $ y \in Y $ such that $ \pbk{c,x,y} \inar \alpha $.
In particular, $ g y = x $, \ie $ x \inar g $.

Suppose now that $ \presubpb{f} h \leq g $.
For every $ i \inar h $ we have $ \presubpb{f} i \leq \presubpb{f} h \leq g $,
so there is $ e \colon Z' \to Y $ such that $ g e = \presubpb{f} i $, where $ Z' $ is the domain of $ \presubpb{f} h $.
It is easy to see that $ \pbk{\presubpb{f} i,e} \colon Z' \to X \times Y $ is a pseudo-relation over $ f,g $ with domain index $ i \in I $.
In particular, there is $ c \in F $ such that $ \phi c = i $, \ie $ i \inar \phi $.
\end{proof}

Recall from \cref{rem:bhklog} that regular logic is valid under the BHK-interpretation in any \qcart category.
From the above lemma and results in~\cite{Pa2004} we obtain the following.

\begin{corol}
Let \catct be a \qcart category which is \prclosd and where every object is a choice object.
Then the $ (\top,\land, \limpd, \exists, \forall) $-fragment of intuitionistic first order logic
is valid under the BHK-interpretation in \catct.
\end{corol}

\begin{rem}
\label{rem:elemlog}
With the same hypothesis as the previous corollary,
we can extend \cref{corol:elemlog}.
Subobjects obtained by $ \overset{\text{\bfseries .}\,}{\limpd} $ and $ \overset{\text{\bfseries .}}{\forall} $ are determined,
up to presubobject equivalence, by the universal closure of the following relations:
\begin{enumerate}
\item[5.]	$ x \inar (a \overset{\text{\bfseries .}\,}{\limpd} b) \ $ \iff $ \ x \inar a \limpld x \inar b $,
\item[6.]	$ x \inar \overset{\text{\bfseries .}}{\forall}_Y r \ $ \iff $ \ (\forall y \in Y)\, \pbk{x,y} \inar r $.
\end{enumerate}
\end{rem}

The following theorem proves that \prclosure provides a sufficient condition
for the local cartesian closure of an elemental exact completion.

\begin{theor}
\label{theor:lccex}
Let \catct  be a \qcart category where every object is a choice object.
If \catct is \prclosd, then \exct is locally cartesian closed.
\end{theor}

The idea of the proof is to isolate the functional relations from a suitable full family of pseudo-relations,
and to define an equivalence relation to identify point-wise equal functional relations.
In doing so, we exploit the characterisations of the internal logic and of the exact completion construction
provided by elementality (\cref{corol:elemlog,corol:elemeqrel,corol:elemexfn,rem:elemlog}).
% 
% The proof exploits the proof-relevance of the BHK-interpretation,
% as well as the characterisations of the internal logic and the exact completion construction
% provided by elementality (\cref{corol:elemlog,corol:elemeqrel,corol:elemexfn,rem:elemlog})
% in order to isolate the functional relations from a suitable full family of pseudo-relations,
% and to define an equivalence relation to identify point-wise equal functional relations.

\begin{proof}
We shall show that \exct has all universal dependent products.
Thanks to \cref{corol:elemeqrel,corol:elemexfn} we can regard objects $ X $ in \exct as pairs $ (X_0, \sim_X) $
where $ \sim_X $ is a pseudo equivalence relation on the elements $ X_0 $,
and arrows $ f \colon X \to Y $ in \exct as arrows $ f \colon X_0 \to Y_0 $ in \catct
such that $ f x \sim_Y f x' $ whenever $ x \sim_X x' $.

Let $ Y \overset{g}{\longrightarrow} X \overset{f}{\longrightarrow} I $ be a pair of composable arrows in \exct.
Define two pseudo-relations $ \tau \colon T_0 \to X_0 \times I_0 $ and $ \sigma \colon S_0 \to Y_0 \times T_0 $ in \catct by the formulas
\begin{equation}
\label{lccex:domcod}
	fx \sim_I i	\qquad\text{and}\qquad		gy \sim_X \tau_1 t,
\end{equation}
respectively, for $ i \in I_0, x \in X_0, y \in Y_0 $ and $ t \in T_0 $,
and let $ \phi \colon F_0 \to I_0 $, $ \alpha \colon P_0 \to F_0 \times T_0 \times S_0 $ be a full family of pseudo-relations
over $ S_0 \overset{\sigma_2}{\longrightarrow} T_0 \overset{\tau_2}{\longrightarrow} I_0 $.
This means that, for every $ c \in F_0, t \in T_0 $ and $ s \in S_0 $,
\begin{gather}
\label{lccex:fullfamsec}	\pbk{c,t,s} \inar \alpha	\limpld		\sigma_2 s = t,	\\
\label{lccex:fullfamdom}	\tau_2 t = \phi c	\liffld		(\exists s \in S_0)\, \pbk{c,t,s} \inar \alpha.
\end{gather}

Let $ \gamma \colon G_0 \to F_0 $ be a presubobject of $ F_0 $ defined by the formula
\begin{equation}
\label{lccex:funrelobj}
	(\forall t,t' \in T_0)(\forall s,s' \in S_0)( \pbk{c,t,s} \inar \alpha	\land	\pbk{c,t',s'} \inar \alpha	%
							    \land	\tau_1 t \sim_X \tau_1 t'	\limpd	\sigma_1 s \sim_Y \sigma_1 s'),
\end{equation}
for $ c \in F_0 $, and let $ \beta \colon Q_0 \to G_0 \times X_0 \times Y_0 $ be the pseudo-relation defined by the formula
\begin{equation}
\label{lccex:funrelrel}
	(\exists t \in T_0)(\exists s \in S_0)( \tau_1 t = x	\land	\sigma_1 s = y	\land	\pbk{\gamma u, t, s} \inar \alpha),
\end{equation}
for $ u \in G_0, x \in X_0 $ and $ y \in Y_0 $.

Define now an equivalence relation $ u \sim_G u' $ on $ G_0 $ as the conjunction of $ \phi \gamma u \sim_I \phi \gamma u' $ and
\begin{equation}
\label{lccex:funrelobjeq}
	(\forall t,t' \in T_0)(\forall s,s' \in S_0)( \pbk{\gamma u,t,s} \inar \alpha \land \pbk{\gamma u',t',s'} \inar \alpha \land%
									\tau_1 t \sim_X \tau_1 t' \limpd \sigma_1 s \sim_Y \sigma_1 s' ).
\end{equation}
Reflexivity follows from \eqref{lccex:funrelobj} and reflexivity of $ \sim_I $, and symmetry is trivial.
To verify transitivity, assume $ u \sim_G u' \sim_G u'' $.
Then $ \phi \gamma u \sim_I \phi \gamma u'' $ follows immediately,
so let $ t,t'' \in T_0 $ and $ s,s'' \in S_0 $ be such that $ \pbk{\gamma u,t,s} \inar \alpha $,
$ \pbk{\gamma u'',t'',s''} \inar \alpha $ and $ \tau_1 t \sim_X \tau_1 t'' $.
We need to show that $ \sigma_1 s \sim_Y \sigma_1 s'' $.
From \eqref{lccex:fullfamdom} and the definition of $ \tau $ in \eqref{lccex:domcod}
we have $ f \tau_1 t \sim_I \tau_2 t = \phi \gamma u \sim_I \phi \gamma u' $,
so there is $ t' \in T_0 $ such that $ \tau t' = \pbk{\tau_1 t, \phi \gamma u'} $
and \eqref{lccex:fullfamdom} yields $ s' \in S_0 $ such that $ \pbk{\gamma u',t',s'} \inar \alpha $.
But we also have $ \tau_1 t = \tau_1 t' $ and $ \tau_1 t' \sim_X \tau_1 t'' $,
hence $ \sigma_1 s \sim_Y \sigma_1 s' $ and $ \sigma_1 s' \sim_Y \sigma_1 s'' $ from \eqref{lccex:funrelobjeq}
and the assumption $ u \sim_G u' \sim_G u'' $.
We have thus established that $ G \coloneqq (G_0,\sim_G) $ is an object in \exct and $ \phi \gamma $ is an arrow $ G \to I $ in \exct.

Define an equivalence relation $ q \sim_Q q' $ on $ Q_0 $ as
\begin{equation}
\label{lccex:funrelreleq}
	\beta_1 q \sim_G \beta_1 q'	\land	\beta_2 q \sim_X \beta_2 q'
\end{equation}
which makes $ Q \coloneqq (Q_0,\sim_Q) $ an object of \exct and $ \beta_1 $ and $ \beta_2 $ arrows $ Q \to G $ and $  Q \to X$, respectively.
We need to check that it also makes $ \beta_3 $ an arrow $ Q \to Y $ in \exct.
For $ q,q' \in Q_0 $ we have, from \eqref{lccex:funrelrel} that there are $ t,t' \in T_0 $ and $ s,s' \in S_0 $ such that
\[	\pbk{\beta_2,\beta_3} q = \pbk{\tau_1 t,\sigma_1 s}	\quad	\pbk{\beta_2,\beta_3} q' = \pbk{\tau_1 t',\sigma_1 s'}	\quad	%
		\pbk{\gamma \beta_1 q, t, s} \inar \alpha	\quad \textup{and} \quad	%
								      \pbk{\gamma \beta_1 q', t', s'} \inar \alpha.	\]
If $ q \sim_Q q' $, then $ \beta_1 q \sim_G \beta_1 q' $ and $ \tau_1 t = \beta_2 q \sim_X \beta_2 q' = \tau_1 t' $,
which in turn imply $ \beta_3 q = \sigma_1 s \sim_Y \sigma_1 s' = \beta_3 q' $ as required.

This gives us a pair of arrows $ \phi \gamma \colon G \to I $ and $ \beta \colon Q \mono G \times X \times Y $ in \exct,
where the latter is monic because of \eqref{lccex:funrelreleq}.
We now need to show that this pair is a universal dependent product for $ f,g $.
Let us first remark that in \exct the membership relation $ \inar $ is different from the one in \catct:
we denote the former with $ \inarex $ and continue denoting the latter as $ \inar $.
In particular, we have:
\begin{equation}
	b \inarex f	\quad \liffld \quad	(\exists b' \in B_0)( b \sim_B b'	\land	b' \inar f ).
\end{equation}

We start showing that the pair $ \phi \gamma,\beta $ is a family of functional relations over $ f,g $
by checking the three properties in \cref{def:properties}.
\begin{enumerate}
\item[\ref{comm}]
	$ \phi \gamma,\beta $ is a family of sections of $ g $:
	If $ \pbk{u,x,y} \inarex \beta $,
	then there are $u' \sim_G u$, $ x' \sim_X x $ and $ y' \sim_Y y $ such that $ \pbk{u',x',y'} \inar \beta $.
	From~\eqref{lccex:funrelrel} we obtain $ t \in T_0 $ and $ s \in S_0 $ such that
	$ \tau_1 t = x' $, $ \sigma_1 s = y' $ and $ \pbk{\gamma u',t,s} \inar \alpha $ and,
	from~\eqref{lccex:fullfamsec}, $ g y \sim_X g \sigma_1 s \sim_X \tau_1 \sigma_2 s \sim_X x $.
\item[\ref{exist}]
	$ \phi \gamma,\beta $ has domains indexed by $ f $:
	Reasoning as above, and using \eqref{lccex:fullfamdom}, it is easy to see that $ \pbk{u,x,y} \inarex \beta $ implies $ f x \sim_I \phi \gamma u $.
	Conversely, let $ x \in X_0 $ be such that $ f x \sim_I \phi \gamma u $.
	From \eqref{lccex:domcod} we obtain $ t \in T_0 $ such that $ \tau t = \pbk{x,\phi \gamma u} $ and,
	from \eqref{lccex:fullfamdom}, we get $ s \in S_0 $ such that $ \pbk{\gamma u,t,s} \inar \alpha $,
	that is, $ \pbk{u,x,\sigma_1 s} \inarex \beta $.
\item[\ref{uniq}]
	$ \phi \gamma,\beta $ is functional:
	Let $ \pbk{u,x,y}, \pbk{u,x',y'} \inarex \beta $ be such that $ x \sim_X x' $.
	As in point~\ref{comm}, we obtain $ t,t' \in T_0 $ and $ s,s' \in S_0 $ such that
	$ \tau_1 t \sim_X x $, $ \tau_1 t' \sim_X x' $, $ \sigma_1 s \sim_Y y $, $ \sigma_1 s' \sim_Y y' $,
	$ \pbk{\gamma u,t,s} \inar \alpha $ and $ \pbk{\gamma u,t',s'} \inar \alpha $.
	Hence $ y \sim_Y y' $ from \eqref{lccex:funrelobj}.
\end{enumerate}

It remains to show that the pair $ \phi \gamma \colon G \to I $,
$ \beta \colon Q \mono G \times X \times Y $ has the required universal property.
Let $ r \colon R \mono X \times Y $ be a functional relation over $ f,g $ with domain index $ i_0 \in I_0 $, \ie such that,
for every $ x,x' \in X_0 $ and $ y,y' \in Y_0 $,
\begin{gather}
\label{lccex:funrelcom'}	\pbk{x,y} \inarex r	\limpld		g y \sim_X x,	\\
\label{lccex:funrelexist'}	f x \sim_I i_0		\liffld		(\exists y \in Y)\, \pbk{x,y} \inarex r,	\\
\label{lccex:funreluniq'}	\pbk{x,y} \inarex r \, \land \, \pbk{x',y'} \inarex r \, \land \, x \sim_X x'	\limpld		y \sim_Y y'.
\end{gather}
Properties \eqref{lccex:funrelcom'} and \eqref{lccex:funrelexist'} above ensure that 
the pseudo-relation $ r' \colon R'_0 \to T_0 \times S_0 $ defined by the formula
\begin{equation}
\label{lccex:psrel}
	 \sigma_2 s = t	\land	\tau_2 t = i_0	\land	\pbk{\tau_1 t,\sigma_1 s} \inarex r.
\end{equation}
is a pseudo-relation over $ \tau_2,\sigma_2 $ with domain index $ i_0 \in I_0 $,
hence from fullness of $ \phi $ and $ \alpha $ we get $ c \in F_0 $ such that $ \phi c = i_0 $ and
\begin{equation}
\label{lccex:fulluniv}
	\pbk{c,t,s} \inar \alpha	\ \limpld \	\pbk{t,s} \inar r'
\end{equation}
for every $ t \in T, s \in S_0 $.

Using \eqref{lccex:funreluniq'}, \eqref{lccex:psrel} and \eqref{lccex:fulluniv}
it is easy to see that $ c \in F_0 $ satisfies \eqref{lccex:funrelobj},
hence there is $ u \in G_0 $ such that $ \gamma u = c $.
We now need to show that
\begin{equation}
\label{lccex:univ}
	\pbk{u,x,y} \inarex \beta	\ \liffld \		\pbk{x,y} \inarex r.
\end{equation}

Suppose $ \pbk{u,x,y} \inarex \beta $,
hence there are $u' \sim_G u$, $ x' \sim_X x $ and $ y' \sim_Y y $ such that $ \pbk{u',x',y'} \inar \beta $.
From~\eqref{lccex:funrelrel} we obtain $ t' \in T_0 $ and $ s' \in S_0 $ such that
$ \tau_1 t' = x' $, $ \sigma_1 s' = y' $ and $ \pbk{\gamma u',t',s'} \inar \alpha $.
On the other hand, the family $ \phi \gamma,\beta $ has domains indexed by $ f $ and, in particular, $ f x \sim_I \phi \gamma u $.
So there are $ t \in T_0 $ such that $ \tau t = \pbk{x,\phi \gamma u} $ and,
from \eqref{lccex:fullfamdom}, $ s \in S_0 $ such that $ \pbk{\gamma u,t,s} \inar \alpha $.
It follows from \eqref{lccex:fulluniv} and \eqref{lccex:psrel} that $ \pbk{x,\sigma_1 s} \inarex r $.
Since $ u \sim_G u' $ and $ \tau_1 t = x \sim_X x' = \tau_1 t' $,
\eqref{lccex:funrelobjeq} implies $ \sigma_1 s \sim_Y \sigma_1 s' = y' \sim_Y y $.
Hence $ \pbk{x,y} \inarex r $.

For the converse, suppose $ \pbk{x,y} \inarex r $.
Then $ f x \sim_I i_0 = \phi \gamma u $ and, since the family $ \phi \gamma,\beta $ has domains indexed by $ f $,
there is $ y' \in Y_0 $ such that $ \pbk{u,x,y'} \inarex \beta $.
But then $ \pbk{x,y'} \inarex r $, and~\eqref{lccex:funreluniq'} implies $ y \sim_Y y' $.
Hence $ \pbk{u,x,y} \inarex \beta $.

It only remains to show uniqueness of $ u \in G $.
Suppose that $ u' \in G_0 $ is such that $ \phi \gamma u' \sim_I i_0 $ and satisfies~\eqref{lccex:univ} for all $ x \in X_0 $ and $ y \in Y_0 $.
Clearly $ \phi \gamma u \sim_I \phi \gamma u' $.
Let $ t,t' \in T_0 $ and $ s,s' \in S_0 $ be such that
$ \pbk{\gamma u,t,s}, \pbk{\gamma u',t',s'} \inar \alpha $ and $ \tau_1 t \sim_X \tau_1 t' $.
Hence $ \pbk{u,\tau_1 t,\sigma_1 s}, \pbk{u',\tau_1 t',\sigma_1 s'} \inarex \beta $ from \eqref{lccex:funrelrel}
and, since both $ u $ and $ u' $ satisfy~\eqref{lccex:univ},
we obtain $ \pbk{\tau_1 t,\sigma_1 s}, \pbk{\tau_1 t',\sigma_1 s'} \inarex r $.
Functionality of $ r $ \eqref{lccex:funreluniq'} implies $ \sigma_1 s \sim_Y \sigma_1 s' $, hence $ u \sim_G u' $ as required.
\end{proof}

%--------------------------------------------------------------------------------------------------------------------------------------------
\section{Models of CETCS as exact completions}											\label{sec:cetcs}

The Constructive Elementary Theory of the Category of Sets (CETCS)
is expressed in a three-sorted language for category theory
and is based on a suitable essentially algebraic formalisation of category theory
over intuitionistic first-order logic.
We refer to~\cite{Pa2012a} for more details.
We now recall the axioms of CETCS.

\begin{enumerate}[label=(C\arabic*),noitemsep]
\item \label{cetcs:c1}		Finite limits and finite colimits exist.
\item \label{cetcs:c2}		Any pair of composable arrows has a universal dependent product.
\item \label{cetcs:c3}		There is a \nno.
\item \label{cetcs:c4}		Elementality.
\item \label{cetcs:c5}		For any object $ X $ there are a choice object $ P $ and a surjection $ P \to X $.
\item \label{cetcs:c6}		The initial object \initt has no elements.
\item \label{cetcs:c7}		The terminal object is indecomposable:
				in any sum diagram $ i \colon X \to S \toop Y \colonop j $,
				$ z \in S $ implies $ z \inar i $ or $ z \inar j $.
\item \label{cetcs:c8}		In any sum diagram $ i \colon \termm \to S \toop \termm \colonop j $,
				the arrows $ i $ and $ j $ are different.
\item \label{cetcs:c9}		Any arrow can be factored as a surjection followed by a mono.
\item \label{cetcs:c10}		Every equivalence relation is a kernel pair.
\end{enumerate}

\begin{rem}
Theorem 6.10 in~\cite{Pa2012a} characterises models of CETCS in terms of standard categorical properties, proving that
CETCS provides a finite axiomatisation of the theory of well-pointed locally cartesian closed pretoposes with a \nno and enough projectives.
Recall that a pretopos is \emph{well-pointed}
if the terminal object is projective, indecomposable, non-degenerate (\ie $ \initm \ncong \termm $) and a strong generator.

In particular, since the terminal object is projective and a strong generator,
we have from \cref{lem:elemconseq} that covers and surjections coincide,
hence the projectives mentioned above are precisely the choice objects given by axiom \ref{cetcs:c5}.
Using cartesian closure, it is also easy to see that these are closed under finite products.
Hence we obtain the following result as a consequence of \cref{theor:carvit}.
\end{rem}

\begin{corol}
\label{corol:projqcart}
Choice objects in a model of CETCS form a \qcart category,
and every model of CETCS is the exact completion of its choice objects.
\end{corol}

\begin{rem}
Choice objects in models of CETCS are in general not closed under all finite limits:
\cref{rem:uip} and \cref{corol:stdcetcs} below show that the category of small types in \mltt
provides a counterexample.
\end{rem}

Using the results in the previous sections,
we can isolate those properties that a \qcart category has to satisfy
in order to arise as a subcategory of choice objects in a model of CETCS.

We begin recalling from~\cite{GrVi1998} that an exact completion \exct is a pretopos \iff \catct has finite sums and is weakly lextensive,
meaning that finite sums interacts well with weak limits.
More precisely,
a \qcart category \catct with finite sums is \emph{weakly lextensive} if
\begin{enumerate}[label=(\alph*),nosep]
\item	sums are disjoint and the initial object is strict,
\item	it is distributive, \ie $ (X \times Y) + (X \times Z) \cong X \times (Y + Z) $,
\item	if $ E_X \to X \psrel Z $ and $ E_Y \to Y \psrel Z $ are weak equalisers,
	then so is $ E_X + E_Y \to X + Y \psrel Z $.
\end{enumerate}
In fact, as observed by Gran and Vitale, the exact completion of a weakly lextensive category coincides with the pretopos completion.

\begin{rem}
\label{rem:nnoex}
Recall that a \emph{\nno} is an object $ N $ together with $ 0 \in N $ and $ s \colon N \to N $ such that,
for any other triple $ X, x \in X, f \colon X \to X $, there is a unique $ g \colon N \to X $ such that
$ g 0 = x $ and $ g s = f s $.
If we drop uniqueness of $ g $, then we obtain a \emph{weak \nno}.

If $ N = (N_0,\sim_N) $ is a \nno in an elemental \exct, then $ N_0 $ is a weak \nno in \catct.
Conversely, a weak \nno in \catct is a weak \nno in \exct as well.
Proposition 5.1 in~\cite{BCRS1998} proves that a cartesian closed category with equalisers and a weak \nno
also has a \nno.
\end{rem}

\begin{prop}
\label{prop:wellpex}
Let \catct be a \qcart category with finite sums.
Then \exct is well-pointed \iff
the terminal object in \catct is non-degenerate and indecomposable and every object in \catct is a choice object.
\end{prop}
\begin{proof}
We already know that the terminal object in \exct is projective and,
from \cref{theor:elemex}, that
it is a strong generator \iff every object in \catct is a choice object.
For non-degeneracy the equivalence follows from the fact that the embedding $ \Gamma \colon \catcm \to \excm $
is conservative and preserves terminal and initial objects.

If the terminal object is indecomposable in \exct, then clearly it is so in \catct as well.
To show the other implication, let us assume elementality of \exct (although it can be easily proved also without it).
Every element $ z \in X + Y $ in \exct is also an element of $ X_0 + Y_0 $ in \catct.
Indecomposability of \termt in \catct implies $ z \inar i_0 $ or $ z \inar j_0 $,
hence we have $ z \inar i $ in the first case, and $ z \inar j $ in the second case,
where $ i_0,j_0 $ (resp.\ $ i,j $) are the coproduct injections of $ X_0 + Y_0 $ (resp.\ of $ X + Y $).
\end{proof}

We can now collect all the properties seen so far which, all together,
provide a sufficient condition ensuring that an exact completion construction will give rise to a model of CETCS.

\begin{theor}
\label{theor:cetcsmod}
Let \catct be a \qcart category with finite sums.
Then \exct is a model of CETCS if
\begin{enumerate}[label={\itshape(\roman*)}]
\item \label{chob}	every object in \catct is a choice object,
\item \label{wp}	the terminal object in \catct is non-degenerate and indecomposable,
\item \label{wlprc}	\catct is weakly lextensive and closed for pseudo-relations,
\item \label{nno}	\catct has a weak \nno.
\end{enumerate}
\end{theor}
\begin{proof}
Well-pointedness of \exct follows from \cref{prop:wellpex},
while Proposition 2.1 in~\cite{GrVi1998} and \cref{theor:lccex} imply that \exct is a locally cartesian closed pretopos.
The existence of enough projectives is automatic from the completion process,
and the existence of a \nno is ensured by Proposition 5.1 in~\cite{BCRS1998} and \cref{rem:nnoex}.
By Theorem 6.10 in~\cite{Pa2012a}, well-pointed locally cartesian closed pretoposes with enough projectives and a \nno
are precisely the models of CETCS.
\end{proof}

%-------------------------------------------------------------------------------------------------------------------------------------
\section{The category of setoids}
\label{sec:setoids}

In this section we apply \cref{theor:cetcsmod} to show that the category of small setoids in \mltt
is a model of CETCS.
We claim no originality on this result:
it is known that this category forms a locally cartesian closed pretopos with a \nno.
In particular, the proof that it is a pretopos has also been formalised in Coq by the second author~\cite{Pa2012b}.
However, we are not aware of a source that presents a complete proof of it,
some references include~\cite{CDPS2005,Ho1994,MoPa2000,Wi2010}.
% The above theorem provides a uniform way to prove that a category is a
% well-pointed locally cartesian closed pretopos with a \nno and enough projectives.

Let \ML be \mltt with rules for $ \sum $-types, $ \prod $-types, identity types $ =_X $, sum types $ \sumty $,
natural numbers type $ \natty $, empty type $ \natty_0 $, one-element type $ \natty_1 $ and
a universe $ (\ttunivm, \mathsf{T}(\cdot)) $ closed under the previous type formers~\cite{NPS90}.
We denote the non-dependent versions of $ \sum $ and $ \prod $ types by $ \prdty $ and $ \fnty $, respectively.
For simplicity, in this presentation we leave the decoding type constructor $ \mathsf{T}(\cdot) $ implicit. 
Henceforth, we shall be working internally in \ML.
% , hence all types and terms have to be considered as in context.
% 
% The constructive notion of set as in Bishop~\cite{Bi1967} may be captured in type theory using setoids.
% A \emph{setoid} consists of a type $ A $ together with 

We shall work with categories without assuming equality on objects.
This formulation is also known as E-category:
% This is realised by defining a category to be formed by 
% The presence of dependent types permits a formulation of the notion of category that avoids equality on objects.
% This formulation is called E-category:
its objects are given by a type, while the arrows between two objects form a setoid,
\ie a type equipped with an equivalence relation which is understood as the equality between arrows.
For more details on E-categories and categories in \mltt we refer to~\cite{PaWi2014}.

The E-category of setoids \catstdt consists of small setoids and extensional functions between them.
More precisely, \emph{small setoids} are pairs $ A \coloneqq (A_0,A_1) $ where
\[	A_0 \typing \ttunivm	\quad \text{and} \quad	A_1 \typing A_0 \fnty A_0 \fnty \ttunivm,	\]
together with proofs of reflexivity, symmetry and transitivity for $ A_1 $.
We write $ A_1(a,a') $ as $ a \sim_A a' $ and omit its proof-terms.
Sometimes we also drop the subscript from the equivalence relation
when this is clear from the context,
and just say ``setoid'' instead of ``small setoid''.

Arrows from a setoid $ A $ to a setoid $ B $ are \emph{extensional functions}, that is,
elements $ f \typing A_0 \fnty B_0 $ together with a proof that $ f $ preserves equality,
\ie an element of
\[	\prod_{a,a' \typing A_0}( a \sim_A a' \fnty f(a) \sim_B f(a') ).	\]
Two extensional functions $ f,g \colon A \to B $ are equal if
\[	\prod_{a \typing A_0} f(a) \sim_B g(a)	\]
is provable:
this is clearly an equivalence relation and makes the type of extensional functions into a (small) setoid.
Identity arrows and composition are defined in the obvious way using application and $\lambda$-abstraction.

Henceforth, we refer to elements of function types as \emph{operations} and to extensional functions simply as functions.
However, we shall not usually distinguish between a function and the underlying operation:
the context should make clear which one we are referring to.

\begin{exmp}
\label{exmp:std}
(1) Every small type $ X \typing \ttunivm $ gives rise to a small setoid $ \frstd{X} \coloneqq (X, =_X) $,
called the \emph{free setoid} on $ X $, where $ x =_X x' $ is the identity type between elements $ x,x' \typing X $.
Proofs of reflexivity, symmetry and transitivity are obtained using the introduction and elimination rules for $ =_X $.

(2) For $ p, n, m \typing \natty $,
define $ n \sim_p m \coloneqq \sum_{k \typing \natty} (n =_\natty m + kp) \sumty (m =_\natty n + kp)  $.
Then when $ p \neq_\natty 0 $,
$ \mathbb{Z}_p \coloneqq (\natty, \sim_p) $ is the setoid of integers modulo $ p $.
If $ q =_\natty lp $,
multiplication by $ l $ on $ \natty $ gives rise to a function $ i \colon \mathbb{Z}_p \to \mathbb{Z}_q $,
and similarly the identity operation $ \lambda x.x \typing \natty \fnty \natty $
produces a function $ e \colon \mathbb{Z}_q \to \mathbb{Z}_p $.
These functions are such that $ e i \sim id_{\mathbb{Z}_p} $.
\end{exmp}

It is not difficult to see that \catstdt has finite limits.
The very definition of setoid ensures that \catstdt has quotients of equivalence relations:
the argument is the same as in Proposition 7.1 in~\cite{MoPa2000}.
The presence of a universe entails that regular epis coincide with surjective functions,
\ie functions $ f \colon A \to B $ such that
\[	\prod_{b \typing B_0} \sum_{a \typing A_0} f(a) \sim_B b	\]
is provable~\cite{Wi2010}.
This implies in particular that they are stable under pullback,
thus we can conclude that \catstdt is exact.

Being inductively defined from reflexivity,
the identity type $ =_X $ is in particular the minimal reflexive relation on the type $ X $.
This has the consequence that, for a setoid $ A $,
any operation $ X \fnty A_0 $ gives rise to a function $ \frstd{X} \to A $ and, in turn,
makes the full subcategory on free setoids a projective cover of \catstdt.
Indeed, for any setoid $ A $,
the identity operation on $ A_0 $ gives rise to a function $ e_A \colon \frstd{A_0} \to A $
which is clearly surjective.
Secondly, given a surjection $ f \colon A \to B $ and a function $ g \colon \frstd{X} \to B $,
the type-theoretic version of the axiom of choice (ttAC), which is provable in type theory
and is also known as the distributivity of $ \prod $-types over $ \sum $-types,
yields an operation $ s \typing X \fnty A_0 $ such that $ f s(x) \sim_B g(x) $ for all $ x \typing X $,
and so a function $ s \colon \frstd{X} \to A $ such that $ f s \sim g $.

\begin{rem}
\label{rem:ttac}
The type-theoretic version of the axiom of choice does not imply that every surjection has a section,
as it instead happens in the category of sets in a model of ZFC.
Indeed, for every surjection $ f \colon A \to B $ in \catstdt,
ttAC does provide an operation $ s \typing B_0 \fnty A_0 $ such that $ f s(b) \sim_B b $ for all $ b \typing B_0 $,
but this is not necessarily extensional.

Consider for example the canonical surjection $ e_{\mathbb{Z}_p} : (\natty, =_\natty) \to \mathbb{Z}_p $
for $ p \neq_\natty 0 $.
Applying ttAC to $ \lambda n. (n + kp,\, \_) \typing \prod_{n \typing \natty} \sum_{m \typing\natty} m \sim_p n $,
where the obvious proof of $ n + kp \sim_p n $ is left implicit,
produces the operation $ i_k \coloneqq \lambda n. n + kp \typing \natty \fnty \natty $, for $ k \typing \natty$.
This operation is clearly not extensional:
for example, $ n \sim_p n + p $ but $ i_k(n) \neq_\natty i_k(n+p) $ for any $ k \typing \natty $.
% 
% We have $ e_{\mathbb{Z}_p} i_k(n) \sim_p n $ for all $ n \typing \natty $ by construction,
% but clearly $ n \sim_p n + p $ and 
% 
% as soon as $ n \neq_\natty n' $,
% $ n \sim_p n' $ does not imply $ i_k(n) =_\natty i_k(n') $ for any $ k \typing \natty $.
% 
% On the other hand it may happen that a suitable proof of surjectivity does provide an extensional function
% upon application of ttAC.
% For instance, this is the case for the functions $ e $ and $ i $ from \cref{exmp:std}(2).

For further details regarding the relation between ttAC and setoids we refer to~\cite{ML2006}.
\end{rem}

The full subcategory of \catstdt on free setoids is in fact equivalent to the E-category of small types \cattypet.
Its type of objects is the universe $ \ttunivm $,
the setoid of arrows $ X \to Y $ has $ X \fnty Y $ as underlying type,
and equivalence relation $ f \sim g $ given by the type
\[	\prod_{x \typing X} f(x) =_Y g(x).	\]

Recall that an E-functor $ F $ between two E-categories \catct and \catdt consists of
an operation $ F_o $ between the types of objects
together with an extensional function $ F_a^{X,X'} \colon Hom_\catcm(X,X') \to Hom_\catdm(F_o(X), F_o(X')) $
for each pair of objects $ X $ and $ X'$ of \catct,
satisfying the usual axioms.
There is an E-functor from \cattypet to \catstdt mapping each small type to the free setoid on it,
and each operation $ f \typing X \fnty Y $ to the corresponding function $ f \colon \frstd{X} \to \frstd{Y} $.
It is clearly fully faithful and its image is the full subcategory on free setoids.
% Conversely, there is an E-functor in the other direction
% which simply forgets the equivalence relation of setoids and proofs of extensionality.
% These two E-functors establish an equivalence between \cattypet and the full subcategory of \catstdt on free setoids.

\cattypet is \qcart:
a product of two objects $ X $ and $ Y $ is given by the (non-dependent) $\sum $-type $ X \prdty Y $,
and a weak equaliser of two arrows $ f,g \colon X \to Y $ is given
by the type $ \sum_{x \typing X} ( f(x) =_Y g(x) ) $ together with the first projection into $ X $.
In addition, the embedding $ X \mapsto \frstd{X} $ preserves all finite products:
since the identity type $ (x,y) =_{X \prdty Y} (x',y') $ is type-theoretically
equivalent to the type $ (x =_X x') \prdty (y =_{Y_0} y') $,
the free setoid on $ X \prdty Y $ is isomorphic to $ \frstd{X} \times \frstd{Y} $.

\begin{rem}
\label{rem:uip}
\cattypet has not arbitrary finite limits,
since their existence would imply the derivability of Uniqueness of Identity Proofs (UIP) in \ML for all small types.
Indeed, given a small type $ X \typing \ttunivm $,
the existence of an equaliser for every pair $ x,x' \colon \termm \to X $
would yield a mere equivalence relation $ E \typing X \fnty X \fnty \ttunivm $
(\ie an equivalence relation such that $ u =_{E(x,x')} v $ for every $ u,v : E(x,x') $ and $ x,x' : X $)
together with an element $ f \typing \Pi_{x,x' : X} E(x,x') \fnty x =_X x' $.
Theorem 7.2.2 in~\cite{HoTTbook} then implies UIP($ X $).
\end{rem}

We may collect what we have seen so far in the following proposition.

\begin{prop}
\catstdt is the exact completion of \cattypet as a \qcart E-category.
\end{prop}

We now verify that \cattypet satisfies the hypothesis of \cref{theor:cetcsmod}.
These are all either immediate or already known in one form or another,
except for closure for pseudo-relations which we prove first.

\begin{prop}
\label{prop:typrc}
\cattypet is closed for pseudo-relations.
\end{prop}
\begin{proof}
Let $ Y \overset{g}{\longrightarrow} X \overset{f}{\longrightarrow} I $ be arrows in \cattypet.
For $ i \typing I $ and $ x \typing X $ define
\[	f^-(i) \coloneqq \sum_{x \typing X} f(x) =_I i,	\qquad \text{and} \qquad	%
							g^-(x) \coloneqq \sum_{y \typing Y} g(y) =_X x,	\]
and form the types
\[	F \coloneqq \sum_{i \typing I} \prod_{u \typing f^-(i)} g^-(\proj_1(u))	\qquad \text{and} \qquad	%
						P \coloneqq \sum_{v \typing F} \sum_{x \typing X} f(x) =_I \phi(v)	\]
where $ \phi \coloneqq \proj_1 \colon F \to I $.
Finally, define $ \epsi \colon P \to Y $ and $ \alpha \colon P \to F \times X \times Y $ as
\[	\epsi(v,x,s) \coloneqq \proj_1 ( (\proj_2 v)(x,s) )	\qquad \text{and} \qquad	%
					\alpha(v,x,s) \coloneqq (v, x, \epsi(v,x,s)).		\]

If $ (v,x,y) \inar \alpha $, then there is $ s \typing f(x) =_I \phi(v) $ such that $ \epsi(v,x,s) =_Y y $ and
\[	\proj_2( (\proj_2 v)(x,s) ) \typing g(\epsi(v,x,s)) =_X x,	\]
so \ref{def:undepprod}\ref{comm} is satisfied,
while \ref{def:undepprod}\ref{exist} follows immediately from the definition of equality of arrows in \cattypet.
Hence the pair $ \phi, \alpha $ is a family of pseudo-relations over $ f,g $.

Let now $ r \colon R \to X \times Y $ be a pseudo-relation over $ f,g $ with domain index $ i \typing I $.
Property \ref{def:undepprod}\ref{exist} implies that
\[	\prod_{u \typing f^-(i)} \sum_{t \typing R} r_1(t) =_X \proj_1(u)	\]
is inhabited,
therefore the type-theoretic axiom of choice yields a function term $ k \typing f^-(i) \to R $ such that
\[	\prod_{u \typing f^-(i)} r_1(k(u)) =_X \proj_1(u).	\]
Property \ref{def:undepprod}\ref{comm} implies that there is a closed term
\[	m \typing \prod_{u \typing f^-(i)} g(r_2(k(u))) =_X \proj_1(u),	\]
Hence we can define a function term $ h \typing \prod_{f^-(i)} g^-(\proj_1(u)) $ as $ h(u) \coloneqq (r_2(k(u)),m(u)) $,
thus obtaining a term $ c \coloneqq (i,h) \typing F $.
Clearly $ \phi(c) =_I i $, we need to show that for all $ x \typing X $ and $ y \typing Y $
\[	(c,x,y) \inar \alpha	\ \limpld \	(x,y) \inar r.	\]

Suppose that there is $ s \typing f(x) =_I \phi(s) $ such that $ (c,x,s) \typing P $ and $ \epsi(c,x,s) =_Y y $,
hence
\[	y =_Y \epsi(c,x,s) =_Y \proj_1 ( h(x,s) ) =_Y r_2(k(x,s)).	\]
Since moreover $ r_1(k(x,s)) =_X \proj_1(x,s) =_X x $,
we can conclude $ (x,y) \inar r $ as required.
\end{proof}

We now briefly recall how to obtain the other properties.

The unit type $ \natty_1 $ gives the terminal object in \cattypet.
It is non-degenerate and indecomposable since the types
\[	(\natty_1 \fnty \natty_0) \fnty \natty_0	\quad \text{and} \quad	%
\prod_{u \typing X \sumty Y} \big ( \sum_{x \typing X} \injl(x) = u \sumty \sum_{y \typing Y} \injr(y) = u \big )	\]
are both inhabited.
Because of ttAC, all objects in \cattypet are choice objects.
% as well as all free setoids in \catstdt.
\cref{lem:weqelem,theor:elemex} then imply that \cattypet and \catstdt are elemental, respectively.

\cattypet is weakly lextensive.
The initial object is given by the empty type $ \natty_0 $
and its elimination rule yields strictness.
Sums are given by sum types $ \sumty $,
and their disjointness follows from the type-theoretic equivalences
\begin{gather*}
	\injl(x) =_{X \sumty Y} \injl(x')	\ \simeq \	x =_X x',	\\
	\injr(y) =_{X \sumty Y} \injr(y')	\ \simeq \	y =_Y y',	\\
	\injl(x) =_{X \sumty Y} \injr(y)	\ \simeq \	\natty_0.
\end{gather*}
See for example~\cite{HoTTbook}.
For distributivity, it is enough to show that
the operation $ (X \prdty Y) \sumty (X \prdty Z) \fnty X \prdty (Y \sumty Z) $
defined by $\sumty$-elimination is injective and surjective,
which is straightforward using the elimination rules of the types involved.
The preservation of weak equalisers also follows from $\sumty$-elimination,
and the fact that a weak equaliser of $ f,g \colon X \to Y $ is logically equivalent over $ X $ to
$ \proj_1 \colon \sum_{x \typing X} f(x) =_Y g(x) \to X $.

The type of natural numbers $ \natty $ provides \cattypet with a \nno.
The existence of a universal arrow is an immediate consequence of the elimination rule of $ \natty $
(\ie recursion on natural numbers)
and, since the equality of arrows in \cattypet is point-wise propositional equality,
such an arrow is in fact unique.

In conclusion we have the following result.

\begin{prop}
\label{corol:stdcetcs}
\cattypet is a \qcart category with finite sums
that satisfies properties \ref{chob}--\ref{nno} from \cref{theor:cetcsmod}.
Hence \catstdt is a model of CETCS.
\end{prop}

%-------------------------------------------------------------------------------------------------------------------------------------
% bibliography
\bibliographystyle{plain}
\bibliography{cetcs-ex-biblio}

\end{document}